%% file: FZ_n.tex
\documentclass[10pt]{amsart}

\usepackage{cite}

\usepackage{amsmath}
\usepackage{amsfonts}
\usepackage{amssymb}
\usepackage{amsthm}
\usepackage{graphicx}
\usepackage{mathrsfs}
\usepackage{manfnt}
\usepackage{color}

\theoremstyle{theorem}
\newtheorem{thm}{Theorem}
\newtheorem{lemma}{Lemma}

\newtheorem*{prethm}{Theorem}
\newtheorem{prop}{Proposition}

\theoremstyle{definition}

\newtheorem*{remark}{Remark}
\newtheorem{ex}{Example}

\newcommand{\Z}{\ensuremath{\mathbb Z}}

\newcommand{\isom}{\cong}

\newcommand{\acts}{\curvearrowright}

\renewcommand{\>}{\ensuremath{\rightarrow}}

\newcommand{\pr}{^{\prime}}

\newcommand{\del}{\partial}

\renewcommand{\part}[2]{\frac{\del{#1}}{\del{#2}}}

\begin{document}

\title{Hyperbolicity of the Cyclic Splitting Complex}

\author{Brian Mann}
\address{\tt Department of Mathematics, University of Utah, 155 S 1400 E, Room 233, Salt Lake City, Utah 84112, USA}
  \email{\tt
  mann@math.utah.edu} 

\maketitle

\begin{abstract} 
We define a new complex on which $Out(F_n)$ acts by simplicial automorphisms, the \emph{cyclic splitting complex of $F_n$}, and show that it is hyperbolic using a method developed by Kapovich and Rafi in \cite{KapovichRafi}.
\end{abstract}

\section{Introduction}

Let $S$ be a hyperbolic surface, perhaps with non-empty boundary or punctures. The \emph{curve complex} $\mathcal{C}(S)$ associated to $S$ is a simplicial complex whose 1-skeleton is defined by taking vertices to be homotopy classes of essential simple close curves in $S$, and where two vertices are joined by an edge if they can be represented by disjoint curves in the surface. 

A celebrated theorem of Masur-Minsky (see \cite{MasurMinsky}) is that
\begin{prethm}[Masur-Minsky]
The curve complex $\mathcal{C}(S)$ is $\delta$-hyperbolic.
\end{prethm}

The mapping class group of $S$, denoted by $Mod(S)$, acts on $\mathcal{C}(S)$ in the obvious manner.

In an attempt to mirror the study of surfaces and their mapping class groups, there have been several complexes defined on which $Out(F_n)$, the group of outer automorphisms of the rank $n$ free group, acts. One such complex is the \emph{free factor complex}, denoted $FF_n$, whose vertices are conjugacy classes of proper free factors and which has the structure of a simplicial complex given by the poset structure on the conjugacy classes of free factors of $F_n$. Bestvina and Feighn have shown in \cite{BestvinaFeighn} that

\begin{prethm}[Bestvina-Feighn]
The free factor complex $FF_n$ is $\delta$-hyperbolic.
\end{prethm} 

Another analogue for the curve complex is the \emph{free splitting complex}, $FS_n$, whose vertices are 1-edge free splittings of $F_n$, and where two vertices are joined by an edge if the two splittings admit a common refinement. A theorem of Handel and Mosher in \cite{HandelMosher} is that

\begin{prethm}[Handel-Mosher]
The free splitting complex $FS_n$ is $\delta$-hyperbolic.
\end{prethm}

A recent paper of Kapovich and Rafi \cite{KapovichRafi} shows that the hyperbolicity of the free splitting complex implies the hyperbolicity of the free factor complex. The outline of their argument goes as follows: the authors define an auxiliary complex $FB_n$ called the \emph{complex of free bases}, whose vertices are conjugacy classes of free bases of $F_n$, and where two conjugacy classes of bases are adjacent if they have representatives which share an element. They show that $FB_n$ and $FF_n$ are quasi-isometric, so it suffices to show that hyperbolicity of $FB_n$ which they do by applying a consequence of Bowditch's work (see Theorem 2 below) in \cite{Bowditch} to the natural map $FS_n\pr \> FB_n$, where $FS_n\pr$ is the barycentric subdivision of $FS_n$. 

\subsection*{The Cyclic Splitting Complex} In this paper, we define another analogue of the curve complex for surfaces, the \emph{cyclic splitting complex}, and show that it is hyperbolic using the technology from \cite{KapovichRafi}.

The cyclic splitting complex is the simplicial complex whose vertices free splittings of $F_n$, and where two free splittings $X$ and $Y$ are connected by an edge if either (1) they are commonly refined (as in $FS_n$) or (2) if there is an element $w$ in the vertex groups of $X$ and $Y$ and a $\Z$-splitting $T$ with edge group $\langle w \rangle$ such that $T$ can be obtained from $X$ and $Y$ by ``folding" $\langle w \rangle$ over the trivial edge groups. There is a natural map $FS_n \> FZ_n$.

It is worth noting why one might wish to examine such a complex. By work of Stallings (see \cite{MS}) a simple closed curve $c$ in a surface $S$ gives a $\Z$-splitting of $\pi_1(S) = A \ast_{\langle c \rangle} B$ or $\pi_1(S) = A \ast_{\langle c \rangle}$, and conversely. The definition of $FZ_n$ is constructed so as to mimic this way of thinking of $\mathcal{C}(S)$.

We show the following:

\begin{prethm}
The cyclic splitting complex $FZ_n$ is $\delta$-hyperbolic.
\end{prethm}

In further work, the author wishes to describe the elements of $Out(F_n)$ which act hyperbolically on $FZ_n$. In particular, a description of these elements should show that $FZ_n$ is not $Out(F_n)$-equivariantly quasi-isometric to the free factor complex.

\subsection*{Acknowledgements} The author thanks Mladen Bestvina, Kasra Rafi, and Patrick Reynolds for their immense patience and for enlightening conversations. 

\section{A Bowditch Hyperbolicity Condition for Graphs }
A \emph{graph} is a connected 1-dimensional simplicial complex. If $X$ and $Y$ are graphs, a $\emph{graph map}$ is a continuous map $f: X \> Y$ such that vertices map to vertices. As always, the vertex set of a graph X is denoted by $V(X)$, and the edge set by $E(X)$. From now on, whenever considering a (connected) simplicial complex $Z$ as a metric space, we mean the  1-skeleton of $Z$ with the simplicial metric. We denote a geodesic path from a vertex $x$ to a vertex $y$ by $[x,y]$. 

In \cite{Bowditch} Bowditch develops a criterion for a graph to be $\delta$-hyperbolic. Similar criteria were applied in the Masur-Minsky proof that the curve complex is hyperbolic. Bowditch defines, for constants $B_1, B_2 > 0$ a \emph{$(B_1,B_2)$-thin triangles structure} in a graph $X$, which is a set of paths $g_{xy}$ between any $x,y \in X$ satisfying some ``thinness" conditions (see \cite{KapovichRafi} for details). Bowditch proves the following useful condition for checking hyperbolicity in \cite{Bowditch}:

\begin{thm}[Bowditch]
Suppose $X$ is a connected graph. If there are $B_1,B_2$ such that $X$ has a $(B_1,B_2)$-thin triangles structure, then there are $\delta > 0$ and $H > 0$ (depending on $B_1, B_2$) such that $X$ is $\delta$-hyperbolic, and every geodesic path from $x$ to $y$ in $X$ is $H$-close to $g_{xy}$. 
\end{thm}

The following proposition, which is proved as a corollary of the above theorem in \cite{KapovichRafi}, will be the main technical tool:

\begin{thm}
Suppose $X$ and $Y$ are connected graphs, $X$ is $\delta$-hyperbolic, and $f : X \> Y$ is $L$-Lipschitz for some $L \geq 0$. Suppose there is $S \subseteq V(X)$ such that
\begin{enumerate}
\item $f(S) = V(Y)$ 
\item $S$ is $D$-dense in $V(X)$ for some $D \geq 0$. 
\item There is an $M > 0$ such that if $x,y \in S$ with $d(f(x),f(y)) \leq 1$ then $diam(f[x,y]) \leq M$. 
\end{enumerate}
Then $Y$ is $\delta\pr$-hyperbolic for some $\delta\pr$.
\end{thm}

\section{The Free Splitting Complex} 

A \emph{tree} is a simply connected graph. An \emph{action} of a group $G$ on a tree $T$, denoted $G \acts T$, is a homomorphism from $G$ to the group of simplicial automorphisms of $T$. An action $G \acts T$ is called \emph{minimal} if there is no proper $G$ invariant subtree of $T$.

Let $F_n$ denote the free group on $n$-generators. Recall from Bass-Serre theory that a minimal action $F_n \acts T$ with trivial edge stabilizers gives a a graph of groups decomposition of $F_n$ with trivial edge groups (and hence a free splitting), and conversely. We shall often refer to the action $F_n \acts T$ as a free splitting, as there will be no confusion. A \emph{k-edge splitting} refers to a free splitting whose associated graph of groups decomposition consists of $k$ edges. Two splittings $F_n \acts T$ and $F_n \acts T\pr$ are \emph{equivalent} if there exists an $F_n$ equivariant homeomorphism $T \> T\pr$.

An equivariant map $f: T \> T\pr$ between minimal $F_n$-trees is called a \emph{collapse map} if the preimage of any point is connected.

Define the \emph{free splitting complex of $F_n$}, denoted $FS_n$ as follows. For a more complete discussion see \cite{HandelMosher}. A vertex of $FS_n$ is an equivalence class of 1-edge free splittings. Two vertices $X,Y \in V(FS_n)$ are connected by an edge if there exists a two edge splitting $T$ and $F_n$-equivariant collapse maps $T \> X$ and $T \> Y$. We say $T$ is a \emph{common refinement} of $X$ and $Y$. A $k$-simplex in $FS_n$ is  a collection of $k+1$ vertices $X_1, \ldots, X_{k+1}$ such that there exists a $k+1$ edge splitting $T$ and $F_n$-equivariant collapse maps $T \> X_i$ for each $i = 1, \ldots, k+1$.

Denote by $FS\pr_n$ the barycentric subdivision of $FS_n$. This is a simplicial complex whose vertices correspond to free splittings of $F_n$, and where there is an edge between vertices $T$ and $T\pr$ if there is an equivariant collapse map $T \> T\pr$ or $T\pr \> T$. The complexes $FS_n$ and $FS\pr_n$ are finite dimensional, connected, and have an action of $Out(F_n)$ by simplicial automorphisms such that the quotient is compact (see \cite{HandelMosher}).

\section{Folding Paths in $FS_n$}

For a general definition of folding paths in $FS\pr_n$, see \cite{HandelMosher}. We need only special types of folding paths between splittings in $CV_n$ which are also discussed in \cite{KapovichRafi}, but we will give an explanation here as well following the treatment there.

Let $T$ be a tree. Vertices of valence $\geq 3$ are \emph{natural vertices}, and connected components of $T \setminus \{\text{natural vertices}\}$ are \emph{natural edges}.

A \emph{rose} $R_n$ is a graph with one vertex and $n$-edges. Given an identification $F_n = \pi_1(R_n)$, a \emph{marking} of a graph $\Gamma$ is a homotopy equivalence $f: \Gamma \> R_n$. Two markings $f: \Gamma \> R_n$ and $f\pr: \Gamma\pr \> R_n$  are equivalent if there is a homeomorphism $\phi:\Gamma\pr \> \Gamma$ such that $f\pr \phi \simeq f\pr$. In particular, an equivalence class of markings gives an isomorphism $\pi_1(\Gamma) \> F_n$.

Let $\beta$ be a basis for $F_n$. As in \cite{KapovichRafi} define a \emph{$\beta$-graph} to be a graph $\Gamma$ with a function $\mu: E(\Gamma) \> \beta \cup \beta^{-1}$ such that if $e$ is an oriented edge of $\Gamma$ and $\bar{e}$ denotes the edge with the opposite orientation, then $\mu(\bar{e}) = \mu(e)^{-1}$. Let $R_\beta$ be the rose whose (oriented) edges are labelled by elements of $\beta$ and their inverses. This labeling gives an identification of $F_n$ with $\pi_1(R_\beta)$.

The labeling of a $\beta$-graph $\Gamma$ determines a map $\Gamma \> R_\beta$ by sending each edge of $\Gamma$ to the edge of $R_\beta$ with the same label. In particular, if this map $\Gamma \> R_\beta$ is a homotopy equivalence, then the labeling of $\Gamma$ gives a marking.

\begin{remark}
A marking of $\Gamma$ corresponds to an action of $F_n$ on the universal cover $\tilde{\Gamma}$, and hence to a point in $FS\pr_n$. Equivalent markings define the same vertex in $FS\pr_n$. We will use $\Gamma$ to denote the vertex in $FS_n\pr$ determined by the marking, hopefully without any confusion.
\end{remark}

\subsection*{Stallings folds} Let $\Gamma$ be a $\beta$-graph such that there exists two edges $e_1$ and $e_2$ with the same initial vertex and such that $\mu(e_1) = \mu(e_2)$. We obtain another $\beta$-graph $\Gamma\pr$ by identifying the edges $e_1$ and $e_2$, and labeling the resulting edge by $\mu(e_1) = \mu(e_2)$. This is called a \emph{Stallings fold} (see \cite{Stallings}). There is a quotient map $\Gamma \> \Gamma\pr$ which is call a \emph{fold map}. Note that if $e_1$ and $e_2$ have distinct terminal vertices, then $\Gamma \> \Gamma\pr$ is a homotopy equivalence.

Suppose that $\Gamma$ is a $\beta$-graph and that the labeling gives a marking $\Gamma \> R_\beta$. If we have two edges in $\Gamma$ with the same initial vertex and label, we can construct another graph $\Gamma\pr$ from $\Gamma$ by a Stallings fold, and the marking $\Gamma \> R_\beta$ factors as $\Gamma \> \Gamma\pr \> R_\beta$. Furthermore, the map $\Gamma\pr \> R_\beta$ is again a marking.

\subsection*{Maximal folds} Suppose $\Gamma$ is a $\beta$-graph and that there exist two edges $e_1$ and $e_2$ in $\Gamma$ with the same initial edge and such that $\mu(e_1) = \mu(e_2)$. Let $\hat e_1$ and $\hat{e_2}$ be natural edges containing $e_1$ and $e_2$. Then $\hat e_1$ and $\hat e_2$ contain maximal initial segments $\tilde e_1$ and $\tilde e_2$ which are labeled by the same word in $\beta$. Therefore we can obtain another graph $\Gamma\pr$ by identifying the segments $\tilde e_1$ and $\tilde e_2 $. We say $\Gamma\pr$ is obtained by a \emph{maximal Stallings fold} or just a maximal fold.

\subsection*{Foldable Maps and Handel-Mosher Folding Paths}

Let $\Gamma$ be a $\beta$-graph, and $f : \Gamma \> R_\beta$ given by the labeling is a marking. We say that the map $f$ is \emph{foldable} if 
\begin{enumerate}
\item For every vertex $v$ of valence 2, the edges $e_1$ and $e_2$ with initial vertex $v$ have $\mu(e_1) \neq \mu(e_2)$.

\item For every natural vertex $v$, there exist three edges $e_1$, $e_2$, and $e_3$ with the same initial vertex $v$ such that $\mu(e_1)$, $\mu(e_2)$ and $\mu(e_3)$ are pairwise unequal. 
\end{enumerate} 

There are a few important properties about foldable maps and maximal folds which we need:
\begin{itemize}
\item A map $\Gamma \> R_\beta$ is foldable in the sense above if and only if the corresponding map between $F_n$-trees $\tilde{\Gamma} \> \tilde{R_\beta}$ is foldable in the sense of Handel-Mosher in \cite{HandelMosher}.
\item If $\Gamma$ is a $\beta$-graph and $\Gamma \> R_\beta$ is foldable, and if $\Gamma\pr$ is obtained from $\Gamma$ by a maximal fold, the induced map $\Gamma\pr \> R_\beta$ is foldable.

\item If $\Gamma \> \Gamma\pr$ is a maximal fold, then $d_{FS\pr_n}(\Gamma,\Gamma\pr) \leq 2$.

\item If $\Gamma \> R_\beta$ is a marking then there exists a finite sequence of maximal folds $$\Gamma = \Gamma_0 \> \Gamma_1 \> \ldots \> \Gamma_N = R_\beta$$

\end{itemize}

The proofs of these are elementary and found in \cite{HandelMosher}. We will also need to following result:

\begin{thm}[Handel-Mosher  \cite{HandelMosher}]
The path in $FS\pr_n$ given by connecting each $\Gamma_i$ and $\Gamma_{i+1}$ by an edge path of length $\leq 2$ is an unparametrized quasi-geodesic in $FS\pr_n$.
\end{thm}

\section{The Cyclic Splitting Complex $FZ_n$}

First, let $F_n \acts T$ be a free splitting. Let $v$ be a vertex of $T$, and let $G_v$ be the stabilizer of $v$ in $F_n = G$. Suppose $w \in G_v$ and let $\langle w \rangle$ denote the cyclic subgroup generated by $w$. Construct a new $F_n$-tree $T$ as follows: choose an edge $e$ with initial vertex $v$. Then for every $\gamma \in G$, identify $\gamma e$ with its orbit under the conjugate $\langle \gamma w \gamma^{-1} \rangle \subseteq G_{\gamma v}$.

The resulting tree $T\pr$ corresponds to a graph of groups decomposition with an edge group $\langle w \rangle$. In particular, if $T$ is a one edge free splitting, then $T\pr$ is a one edge splitting with cyclic edge group. See figures 1 and 2 below.

We say $T\pr$ is obtained from $T$ by an equivariant \emph{edge fold}. The natural map $T \> T\pr$ is called an \emph{edge folding map}.

\begin{ex}
Suppose $F_4 = \langle a,b,c,d \rangle$. Consider the one-edge free splitting $A \ast B$ given by $A = \langle a,b \rangle$ and $B = \langle c,d \rangle$. Then the one edge $\Z$-splitting $A \ast_{\langle [a,b] \rangle} \langle B,[a,b]\rangle$ is obtained from $A \ast B$ by an edge fold.
\end{ex}

It is a theorem of Bestvina and Feighn (see Lemma 4.1 in \cite{OuterLimits}) that any $\Z$-splitting can be ``unfolded."

\begin{thm}[Bestvina-Feighn \cite{OuterLimits}]
Let $\Gamma$ be a graph of groups decomposition of the free group $F_n$ with cyclic edge groups. Then either all of the edge groups of $\Gamma$ are trivial or there exists an edge $e$ with stabilizer $G_e \isom \Z$ and a vertex $v$ which is an endpoint of $e$ such that the inclusion $i : G_e \> G_v$ has image a free factor of $G_v$ and furthermore for any edge $e\pr$ incident at $v$ the image of $G_{e\pr} \> G_v$ lies in a complementary free factor.
\end{thm}

In particular, Theorem 4 generalizes an earlier theorem of Shenitzer, Stallings, and Swarup (see \cite{Shenitzer}, \cite{StallingsGtrees}, \cite{Swarup}) that any $\Z$-splitting $A \ast_\Z B$ is obtained by edge folding from a free splitting as in the above example. \\

Define a complex $FZ_n$, the \emph{cyclic splitting complex of $F_n$} as follows. The vertices of $FZ_n$ are 1-edge free splittings of $F_n$. Free splittings $X$ and $Y$ are connected by an edge if
\begin{itemize}
\item  there exists a 2-edge splitting and $F_n$-equivariant collapse maps $T \> X$ and $T \> Y$.
\item  there exists a $\Z$-splitting $T$ and equivariant edge folds $X \> T$, $Y \> T$.
\end{itemize}

Distinct vertices $X_1, \ldots, X_{k+1}$ define a $k$-simplex if each $X_i$ and $X_j$ with $i \neq j$ are pairwise adjacent. 

Note that there is a natural inclusion $i: FS_n \> FZ_n$. If two free splittings are connected by an edge in $FS_n$, then their images are also connected by an edge of the first type in $FZ_n$. $Out(F_n)$ acts on $FZ_n$ by simplicial automorphisms in the obvious way.

We can now extend this map $i$ to a map $f$ from the barycentric subdivision of $FS_n$ to $FZ_n$ as follows: A vertex $V$ of $FS_n\pr$ is a $k$-edge splitting of $F_n$. Define $f(V)$ to be the splitting obtained by collapsing all edges but one to a point. The map is only coarsely well-defined, but for any choice of edge in $V$, the 1-edge splittings obtained will be at most distance 1 apart. Then extend to a graph map from $FS_n\pr \> FZ_n$.

Furthermore, $f$ restricts to $i$ on the vertices of $FS_n$ (these are already 1-edge splittings), and $f$ is clearly 1-Lipschitz as well. We will need the following useful lemma.

\begin{figure}
	
	\centering
	\def\svgwidth{\columnwidth}
	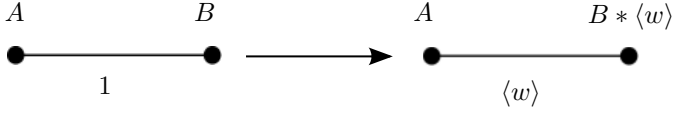
	\caption{Folding a two-vertex splitting}
\end{figure}

\begin{figure}
	\centering
	\def\svgwidth{300pt}
	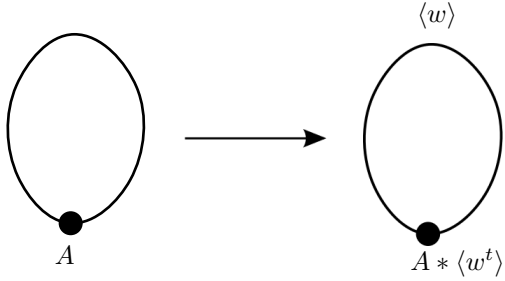
	\caption{Folding a one-vertex splitting. The non-trivial loop represents the element $t$.}
\end{figure} 

\begin{lemma}
Suppose $X$ and $Y$ are one-edge, two-vertex splittings connected by an edge of the second type in $FZ_n$. Then there exist vertices $X\pr$ and $Y\pr$ with $d(X,X\pr) \leq 1$ and $d(Y,Y\pr) \leq 1$ such that $d(X\pr,Y\pr)\leq 1$ and which share a vertex group.
\end{lemma}

\begin{proof}
Let $\langle w \rangle$ be the edge group of the $\Z$-splitting to which $X$ and $Y$ fold, and let $A$ be the smallest free factor containing $\langle w \rangle$ (in which case, we say that the element $w$ \emph{fills} $A$). Then there exist $X\pr = A \ast B$ and $Y\pr = A \ast B\pr$ such that $X$ and $X\pr$, and $Y$ and $Y\pr$ share a common refinement.
\end{proof}

\section{Hyperbolicity of $FZ_n$}

We use the map $f : FS_n\pr \> FZ_n$ and the method pioneered in \cite{KapovichRafi} to prove the following theorem:

\begin{thm}
The cyclic splitting complex $FZ_n$ is $\delta$-hyperbolic.
\end{thm}

\begin{proof}
Let $S$ be the set of $1$-edge splittings in $FS_n$. Conditions $(1)$ and $(2)$ of Theorem 2 are clearly satisfied.

Since $FS_n$ is $\delta$-hyperbolic by \cite{HandelMosher}, by Theorem 2 it suffices to show condition (3) is true: that there exists an $M > 0$ such that for any 1-edge free splittings $X$ and $Y$, if $d_{FZ_n}(f(X),f(Y)) \leq 1$, then $diam(f[X,Y]) \leq M$. 

Suppose $X$ and $Y$ are 1-edge free splittings of $F_n$ which are joined by an edge of the second type in $FZ_n$ (note that it suffices to cover this case because an edge of type 1 corresponds to being distance 1 in $FS\pr_n$). Suppose $T$ is the $\Z$-splitting such that there exist edge folds $X \> T$ and $Y \> T$. 

There are two cases to cover: (1) the graph of groups of the $\Z$-splitting $T$ is a segment and (2) is a loop:
\subsubsection*{Case 1}
By lemma 1, choosing splittings at most distance 1 away we may assume $X = A \ast B$ and $Y = A \ast B\pr$, and that the element $w$ fills $A$. Then the condition that $X$ and $Y$ fold to $T$ is exactly the condition that $B \ast \langle w \rangle = B\pr \ast \langle w \rangle$. Let $R$ and $R\pr$ be free splittings defined as follows: both have underlying graphs which are roses, and the loops of $R$ represent elements of bases of $A$ and $B$; in particular choose a basis $\{a_1, \ldots, a_k\}$ of $A$ and a basis $\{b_{1}, \ldots, b_l\}$ of $B$ and label the edges of $R$ by the collection of $a_i$'s and $b_j$'s. Denote the basis of $F_n$ formed by this collection by $\beta$. 

Choose a basis $\{b_1\pr, \ldots, b_l\pr \}$ for $B\pr$ and define $R\pr$ as the rose whose edges represent the elements in the basis $\beta\pr = \{a_1, \ldots, a_k, b_1\pr, \ldots b_l\pr \}$. Label the edges of $R\pr$ by the elements of $\beta\pr$ written in the basis $\beta$ (subdividing the edges of $R\pr$ as necessary) so that both are $\beta$-graphs. Note that we see a subgraph labelled by the $a_i$'s in both. This gives a homotopy equivalence $R\pr \> R$. By perhaps conjugating, we may assume that the map $R\pr \> R$ is foldable; indeed $R\pr \> R$ fails to be foldable exactly when the word labeling each edge starts and ends with some $a_i$, so by conjugating every label we can be sure this does not happen without changing the splitting. Note that $R$ (resp. $R\pr)$ has $d(f(R),X) \leq 1$ (resp. $d(f(R\pr),Y) \leq 1$).

Then choose a Handel-Mosher folding path $R\pr = \Gamma_0 \> \Gamma_1 \> \cdots \> \Gamma_N = R$ as follows: recall that $B \ast \langle w \rangle = B\pr \ast \langle w \rangle$ so that the basis elements $b\pr_1, \ldots, b\pr_j$ written in terms of $\beta$ are just words in $w$ and $b_1, \ldots, b_l$. Each maximal fold $\Gamma_i \> \Gamma_{i+1}$ occurs as one of the following two types, either (1) fold a loop labelled by some $b\pr_j$ over a letter $a_i$ in the word $w$ or (2) fold maximal initial segments of two distinct loops labelled by some $b_i$ and $b_j$. We require to fold an entire $w$ before moving on to a fold of the second type: if we do a maximal fold of type 1 and fold only a proper subword of $w$, then after this fold we still see natural edges with the same initial label. By \cite{HandelMosher}, regardless of the order in which the edges are folded, we still end up at $R$, so we continue doing maximal folds until we have folded out the entire word $w$.

In particular, each type of fold (1) or (2) either leaves the splitting $f(\Gamma_{i+1})$ (coarsely) equal to $f(\Gamma_{i})$ or it gives another splitting $A \ast B_i$ within distance $1$ of $f(\Gamma_i)$ so that $B_i \ast \langle w \rangle = B\pr \ast \langle w \rangle$. In either case, at each step of the folding path $\Gamma_i$, we have $d(f(\Gamma_i), X) \leq 3$.

\subsubsection*{Case 2} Suppose the vertex group of $X$ is $A \ast B$ and the vertex group of $Y$ is $A \ast B\pr$, where $A$ is the smallest free factor containing $\langle w \rangle$. $X$ and $Y$ are adjacent to a common $\Z$-splitting $T$ exactly when $A \ast B \ast \langle w^t \rangle = A \ast B\pr \ast \langle w^t \rangle$, where $t$ is the element of $F_n$ corresponding to the non-trivial loop in the graph of groups, and $\langle w \rangle$ is the edge group of $T$. 

We follow the same basic outline as in the segment case: choose splittings $R$ and $R\pr$ as follows: Let $\{a_1, \ldots, a_k\}$ be a basis of $A$, and $\{b_{1}, \ldots, b_l\}$ a basis of $B$. Let $R$ be the splitting with underlying graph a rose and whose edges are labelled by the elements in the basis $\beta = \{a_1, \ldots, a_k, t, b_{1}, \ldots, b_l\}$. Choose a basis $\{b\pr_1, \ldots, b\pr_l\}$ for $B\pr$ and let $\beta\pr = \{a_1, \ldots, a_k, t, b\pr_{1}, \ldots, b\pr_l\}$. We choose $R\pr$ to be the rose whose edges are represent the elements in the basis $\beta\pr$. Label the edges of $R\pr$ by this elements of the basis $\beta\pr$ written in the letters of $\beta$ so that both $R$ and $R\pr$ become $\beta$-graphs.

By conjugating, we may assume that the homotopy equivalence $R\pr \> R$ given by the markings is foldable. Also recall that $A \ast B \ast \langle w^t \rangle = A \ast B\pr \ast \langle w^t \rangle$ so that every basis element $b\pr_i$ is written as a word in $A \ast B \ast \langle w^t \rangle$. Thus, there is a Handel-Mosher folding path $R\pr \> \Gamma_0 \> \Gamma_1 \> \ldots \> \Gamma_N = R$ all of whose maximal folds $\Gamma_i \> \Gamma_{i+1}$ either (1) fold an edge labelled by some $b\pr_i$ over a edge whose label is a letter in the word $w^t$ or (2) folds an initial segment of an edge labelled by $b_i$ with the initial segment of a different edge labelled by $a_j$ or $b_j$. As above, with the first type of fold we make sure to fold an entire word $w^t$ before moving on. In both cases, either the splitting $f(\Gamma_{i+1})$ is within distance 1 of $f(\Gamma_{i})$ or $f(\Gamma_i)$ is distance 1 from a  1-edge splitting whose vertex group is of the form $A \ast B_i \ast \langle w^t \rangle$.

In either case, for each $i$ we have $d(f(\Gamma_i),X) \leq 3$. \\

By \cite{HandelMosher}, since each of the maps $\Gamma_i \> \Gamma_{i+1}$ are maximal folds, $d(\Gamma_i,\Gamma_{i+1}) \leq 2$. Since the map $f$ is Lipschitz, this implies that the image of this folding path is contained in a bounded neighborhood of the splitting $T$. Furthermore, since folding paths are unparametrized quasi-geodesics by Theorem 3, this path is uniformly close to $[R\pr,R]$, and because $FS\pr_n$ is hyperbolic, the geodesic $[X,Y]$ is uniformly close to $[R\pr,R]$. Putting all this together, we see that there exists a constant $M \geq 0$ such that $diam(f[X,Y]) \leq M$.
\end{proof}

\section{Other definitions}

There are a few other candidate complexes that we might have called the cyclic splitting complex. We will show that all these complexes are $Out(F_n)$-equivariantly quasi-isometric. 

Define the complex $\overline{FZ}_n$ as follows: vertices of $\overline{FZ}_n$ are one edge free or $\Z$-splittings of $F_n$. Two such vertices $X$,$Y$ are connected if (1) $X$ and $Y$ are free splittings which admit a common refinement or (2) $X$ can be obtained from $Y$ by an edge fold.

\begin{prop}
$FZ_n$ and $\overline{FZ}_n$ are $Out(F_n)$-equivariantly quasi-isometric. 
\end{prop}

\begin{proof}
Define a map $\phi : FZ_n \> \overline{FZ}_n$ in the obvious way: send a vertex of $FZ_n$ to the corresponding free splitting in $\overline{FZ}_n$. Extend to a map of the entire complex. $\phi$ is clearly equivariant.

Let $X,Y \in V(FZ_n)$ with $d(X,Y) \leq 1$. Then by definition of $\overline{FZ}_n$, at worst $d(\phi(X),\phi(Y)) \leq 2$. Hence $d(X,Y) \leq 2d(\phi(X),\phi(Y))$. Furthermore, if $d(f(X),f(Y)) \leq 2$, then $X$ and $Y$ are joined by a path of length at most 2. In particular, $$\frac{1}{2}d(X,Y) \leq d(\phi(X),\phi(Y)) \leq 2d(X,Y)$$
so $\phi$ is quasi-isometry as desired.
\end{proof}

There is third complex whose definition more closely resembles the definition of $FS_n$. Define a complex $C_n$ whose vertices are 1-edge free or $\Z$-splittings and where two vertices $X$ and $Y$ are connected by an edge if the corresponding splittings have a 2-edge common refinement.

\begin{prop}
$FZ_n$ and $C_n$ are $Out(F_n)$-equivariantly quasi-isometric.
\end{prop}

\begin{proof}
We will actually show that there is an $Out(F_n)$-equivariant quasi-isometry $\phi: \overline{FZ}_n \> C_n$. The vertex sets of $\overline{FZ}_n$ and $C_n$ are the same, so set $\phi$ to be the identity on vertices. Then extend $\phi$ to a map of graphs.

Let $X$ and $Y$ be vertices of $\overline{FZ}_n$ such that $d(X,Y) \leq 1$. Then by folding the edge group $\langle w \rangle$ ``half-way" over the edge of the splitting, we get a two-edge splitting which commonly refines both $X$ and $Y$. More precisely, if $X$ and $Y$ are one-vertex splittings, consider the 2-edge splitting in which the edges are adjacent at both endpoints, the vertex groups are the vertex group of $X$ and the edge group $\langle w \rangle$ of the $\Z$-splitting, and the edge groups are trivial and $\langle w \rangle$. 

Similarly, if $X$ and $Y$ are two-vertex splittings, say $X = A \ast  B$ and $Y = A \ast_{\langle w \rangle} \langle B, w \rangle$, then we can consider the two edge splitting $A \ast_{\langle w\rangle} \langle w \rangle \ast B$ which refines $X$ and $Y$.
Hence $d(\phi(X),\phi(X)) \leq d(X,Y)$.

Now suppose $X$ and $Y$ are cyclic splittings which are commonly refined by a two-edge splitting (so $d(\phi(X),\phi(Y)) \leq 1$). We need to find a uniform $L>0$ such that $d(X,Y) \leq L$. 

Suppose that $X$ is a $\Z$-splitting and $Y$ is a free splitting. Then we can unfold the edge group from $X$ to get a free splitting $X\pr$. Since $X$ and $Y$ are commonly refined, writing down the vertex groups of $X\pr$ we see that $X\pr$ is commonly refined with $Y$, and hence $d(X,Y) \leq 2$ in $\overline{FZ}_n$. We will do one case carefully - the others are similar and left to the reader. Suppose the graphs of groups corresponding to both $X$ and $Y$ are segments. Then the common refinement is $A \ast_{\langle w \rangle} B \ast C$, with $X = A \ast_{\langle w \rangle} (B \ast C)$ and $Y = (A \ast_{\langle w \rangle} B) \ast C$.

Suppose furthermore that $A \ast_{\langle w \rangle} (B \ast C)$ unfolds to $X\pr = A \ast (B\pr \ast C)$. Then $X\pr$ and $Y$ are commonly refined by the two-edge splitting $A \ast B\pr \ast C$ since the factors $A \ast_{\langle w \rangle} B$ and $A \ast B\pr$ are equal. The other cases are similar, so we have the above inequality for $L = 2$. 

Suppose both $X$ and $Y$ are $\Z$-splittings which are commonly refined by a two-edge splitting, say $A \ast_{\langle s \rangle } B \ast_{\langle t\rangle} C$, so in particular $X = (A \ast_{\langle s \rangle } B) \ast_{\langle t\rangle} C$ and $Y = A \ast_{\langle s \rangle} (B \ast_{\langle t\rangle} C)$. By Theorem 4 above, one of the edge groups $\langle s \rangle$ or $\langle t \rangle$ can be unfolded to get a splitting with one trivial edge stabilizer, and one $\Z$-stabilizer. We can then unfold the remaining $\Z$-edge to get a free splitting, say $A\pr \ast B\pr \ast C\pr$ which is a common refinement of the 1-edge free splittings $(A\pr \ast B\pr) \ast C\pr$ and $A\pr \ast (B\pr \ast C\pr)$. Furthermore, we can fold the element $t$ over the edge of the splitting $(A\pr \ast B\pr) \ast C\pr$ to get a $\Z$-splitting equal to $X$ and folding $s$ over the edge of $A\pr \ast (B\pr \ast C\pr)$ we obtain the splitting $Y$.

Hence, $d(X,Y) \leq 3$. Therefore, we have $$\frac{1}{3}d(X,Y) \leq d(\phi(X),\phi(Y)) \leq 3d(X,Y)$$
so $\phi: \overline{FZ}_n \> C_n$ is a quasi-isometry.
\end{proof}

\section*{Remark}

There are natural $Out(F_n)$-equivariant maps: $FS_n \> FZ_n$ discussed above and $FZ_n \> FF_n$, which is given by sending a vertex of $FZ_n$ to one of the edge groups of the corresponding free splittings. There is also a natural map $FS_n \> FF_n$ defined in the same way, which clearly factors through the above maps. It is a priori unclear that these maps are not quasi-isometries.

\begin{prop}
For $n \geq 3$, the natural $Out(F_n)$-equivariant map $FS_n \> FZ_n$ is not a quasi-isometry.
\end{prop} 

\begin{proof}
We assume the following result claimed in \cite{HandelMosher}, and which is to appear in second part of their work on the free splitting complex: An element $\phi \in Out(F_n)$ acts hyperbolically on $FS_n$ is there exists an attracting lamination $\Lambda$ of $\phi$ whose support is all of $F_n$, i.e. $\Lambda$ is not carried by a proper free factor. 

Let $S$ be a surface so that $\pi_1(S)$ is free of rank $n \geq 3$, and suppose there exists a (non-separating) simple closed curve $C$ in $S$ such that $S \setminus C = \Sigma$ admits a pseudo-Anasov mapping class. $C$ gives a $\Z$-splitting $\pi_1(S) = A\ast_C$. Such a surface exists for all $n \geq 3$: indeed, for $n \geq 4$, we can take a surface of genus $\geq 2$ with either one or two boundary components. To get $n = 3$, take a twice-punctured torus. We can cut along a curve $C$ to get a 4-times punctures sphere, which admits a pseudo-Anasov homeomorphism.

Let $\phi$ be a pseudo-Anasov mapping class of $\Sigma$. Hence $\phi$ can also be thought of as an outer automorphism of $\pi_1(S)$ which fixes the splitting $A\ast_C$. It remains to show that the expanding lamination $\Lambda$ of $\phi$ is not carried by a proper free factor. Since $\phi$ is pseudo-Anasov, it follows that $\Lambda$ restricted to $\Sigma$ is minimal and filling. 

Suppose not, and let $H$ be a factor which carries a leaf $L$ of $\Lambda$ (and hence carries all of $\Lambda$ by minimality). Let $\tilde{S}$ be the cover of $S$ corresponding to $H$. Since $H$ is finitely generated, there exists a compact subsurface $S_H$ of $\tilde{S}$ which has fundamental group $H$. Let $\tilde{L}$ be a lift of $L$ to $\tilde{S}$ which is contained in $S_H$. Let $\Sigma\pr$ be the smallest subsurface of $S_H$ which $\tilde{L}$ fills. Suppose $\tilde{L}_0$ is any other lift of $L$ which meets $\Sigma\pr$. If $\tilde{L}_0$ is not contained in $\Sigma\pr$, it enters through some boundary component. By Theorem 5.2 in \cite{FLP} (this is for foliations, but the result for laminations is similar), $\tilde{L}_0$ exits through a boundary component of $\Sigma\pr$. In particular, by cutting $\Sigma\pr$ along $\tilde{L}_0$ we would obtain a smaller surface which $\tilde{L}$ fills, a contradiction. Hence, any lift of $L$ meeting $\Sigma\pr$ must be contained in $\Sigma\pr$.

In particular, $\Sigma\pr$ is (homotopic to) a finite cover of $\Sigma$, so $H$ contains a finite index subgroup of $\pi_1(\Sigma) = A$, which is impossible unless $H = F_n$. Indeed, $\pi_1(\Sigma)$ cannot be contained in a proper free factor; looking at Euler characteristics we see $\chi(S) = \chi(\Sigma)$ so $\chi(\tilde{\Sigma}) \leq \chi(S)$, so the rank of $\pi_1(\tilde{\Sigma})$ is at least $n$.
\end{proof}

\bibliographystyle{amsplain}
\bibliography{Refs}{}

\end{document}

%% file: folding.pdf_tex
\begingroup%
  \makeatletter%
  \providecommand\color[2][]{%
    \errmessage{(Inkscape) Color is used for the text in Inkscape, but the package 'color.sty' is not loaded}%
    \renewcommand\color[2][]{}%
  }%
  \providecommand\transparent[1]{%
    \errmessage{(Inkscape) Transparency is used (non-zero) for the text in Inkscape, but the package 'transparent.sty' is not loaded}%
    \renewcommand\transparent[1]{}%
  }%
  \providecommand\rotatebox[2]{#2}%
  \ifx\svgwidth\undefined%
    \setlength{\unitlength}{1046.520419bp}%
    \ifx\svgscale\undefined%
      \relax%
    \else%
      \setlength{\unitlength}{\unitlength * \real{\svgscale}}%
    \fi%
  \else%
    \setlength{\unitlength}{\svgwidth}%
  \fi%
  \global\let\svgwidth\undefined%
  \global\let\svgscale\undefined%
  \makeatother%
  \begin{picture}(1,0.11059865)%
    \put(0,0){\includegraphics[width=\unitlength]{folding.pdf}}%
    \put(-0.00254383,0.11784275){\color[rgb]{0,0,0}\makebox(0,0)[lt]{\begin{minipage}{0.09684044\unitlength}\raggedright $A$\end{minipage}}}%
    \put(0.19641927,0.11784275){\color[rgb]{0,0,0}\makebox(0,0)[lt]{\begin{minipage}{0.10036191\unitlength}\raggedright $B$\end{minipage}}}%
    \put(0.42707562,0.11784275){\color[rgb]{0,0,0}\makebox(0,0)[lt]{\begin{minipage}{0.12325149\unitlength}\raggedright $A$\end{minipage}}}%
    \put(0.61019213,0.11784275){\color[rgb]{0,0,0}\makebox(0,0)[lt]{\begin{minipage}{0.44898752\unitlength}\raggedright $B \ast \langle w \rangle$\end{minipage}}}%
    \put(0.51863385,0.03684891){\color[rgb]{0,0,0}\makebox(0,0)[lt]{\begin{minipage}{0.37327587\unitlength}\raggedright $\langle w \rangle$\end{minipage}}}%
    \put(0.09605735,0.04037039){\color[rgb]{0,0,0}\makebox(0,0)[lt]{\begin{minipage}{0.10036191\unitlength}\raggedright $1$\end{minipage}}}%
  \end{picture}%
\endgroup%

%% file: folding2.pdf_tex
\begingroup%
  \makeatletter%
  \providecommand\color[2][]{%
    \errmessage{(Inkscape) Color is used for the text in Inkscape, but the package 'color.sty' is not loaded}%
    \renewcommand\color[2][]{}%
  }%
  \providecommand\transparent[1]{%
    \errmessage{(Inkscape) Transparency is used (non-zero) for the text in Inkscape, but the package 'transparent.sty' is not loaded}%
    \renewcommand\transparent[1]{}%
  }%
  \providecommand\rotatebox[2]{#2}%
  \ifx\svgwidth\undefined%
    \setlength{\unitlength}{911.99975886bp}%
    \ifx\svgscale\undefined%
      \relax%
    \else%
      \setlength{\unitlength}{\unitlength * \real{\svgscale}}%
    \fi%
  \else%
    \setlength{\unitlength}{\svgwidth}%
  \fi%
  \global\let\svgwidth\undefined%
  \global\let\svgscale\undefined%
  \makeatother%
  \begin{picture}(1,0.34511063)%
    \put(0,0){\includegraphics[width=\unitlength]{folding2.pdf}}%
    \put(0.06105061,0.04632674){\color[rgb]{0,0,0}\makebox(0,0)[lt]{\begin{minipage}{0.07071559\unitlength}\raggedright $A$\end{minipage}}}%
    \put(0.50958946,0.04228587){\color[rgb]{0,0,0}\makebox(0,0)[lt]{\begin{minipage}{0.52531581\unitlength}\raggedright $A \ast \langle w^t \rangle$\end{minipage}}}%
    \put(0.51767126,0.35343443){\color[rgb]{0,0,0}\makebox(0,0)[lt]{\begin{minipage}{0.36368011\unitlength}\raggedright $\langle w \rangle$\end{minipage}}}%
  \end{picture}%
\endgroup%